\documentclass[reqno]{amsart}
\usepackage{amsmath}
\usepackage{amsfonts}

\newtheorem{theorem}{Theorem}
\theoremstyle{plain}

\newtheorem{corollary}{Corollary}

\newtheorem{proposition}{Proposition}

\numberwithin{equation}{section}

\begin{document}
\title[On Slant Magnetic Curves in $S$-manifolds]{On Slant Magnetic Curves
in $S$-manifolds}
\author{\c{S}ABAN G\"{U}VEN\c{C}}
\address[\c{S}. G\"{u}ven\c{c} and C. \"{O}zg\"{u}r]{Department of
Mathematics, Balikesir University, 10145, \c{C}a\u{g}\i \c{s}, Bal\i kesir,
TURKEY}
\email[\c{S}. G\"{u}ven\c{c}]{sguvenc@balikesir.edu.tr}
\author{C\.IHAN \"{O}ZG\"{U}R}
\email[C.~\"{O}zg\"{u}r]{cozgur@balikesir.edu.tr}
\date{}

\begin{abstract}
We consider slant normal magnetic curves in $(2n+1)$-dimensional $S$%
-manifolds. We prove that $\gamma $ is a slant normal magnetic curve in an $%
S $-manifold $(M^{2m+s},\varphi ,\xi _{\alpha },\eta ^{\alpha },g)$ if and
only if it belongs to a list of slant $\varphi $-curves satisfying some
special curvature equations. This list consists of some specific geodesics,
slant circles, Legendre and slant helices of order $3$. We construct slant
normal magnetic curves in $%
\mathbb{R}
^{2n+s}(-3s)$ and give the parametric equations of these curves.
\end{abstract}

\maketitle

\section{Introduction\label{introduction}}

\bigskip Let $(M,g)$ be a Riemannian manifold, $F$ a closed $2$-form and let
us denote the Lorentz force on $M$ by $\Phi $. \ If $F$ is associated by the
relation%
\begin{equation}
g(\Phi X,Y)=F(X,Y),\text{ \ }\forall X,Y\in \chi (M),  \label{F}
\end{equation}%
then it is called a\textit{\ magnetic field} (\cite{Adachi-1996}, \ \cite%
{BRCF} and \cite{Comtet-1987}). Let $\nabla $ be the Riemannian connection
associated to the metric $g$ and $\gamma :I\rightarrow M$ a smooth curve. If
$\gamma $ satisfies the Lorentz equation%
\begin{equation}
\nabla _{\gamma ^{\prime }(t)}\gamma ^{\prime }(t)=\Phi (\gamma ^{\prime
}(t)),  \label{Lorentz eq}
\end{equation}%
then it is called a \textit{magnetic curve} or a \textit{trajectory} for the
magnetic field $F$. The Lorentz equation is a generalization of the equation
for geodesics. A curve which satisfies the Lorentz equation is called
\textit{magnetic trajectory}. Magnetic trajectories have constant speed. If
the speed of the magnetic curve $\gamma $ is equal to $1$, then it is called
a \textit{normal magnetic curve }\cite{DIMN-2015}.

In \cite{Adachi-1996}, Adachi studied curvature bound and trajectories for
magnetic fields on a Hadamard surface. He showed that every normal
trajectory is unbounded in both directions in a $2$-dimensional complete
simply connected Riemannian manifold satisfying some special curvature
conditions. In \cite{CB-1994}, Baikoussis and Blair considered Legendre
curves in contact 3-manifolds and they proved that the torsion of a Legendre
curve in a 3-dimensional Sasakian manifold is equal to $1$. Moreover, in
\cite{CIL}, Cho, Inoguchi and Lee proved that a non-geodesic curve in a
Sasakian 3-manifold is a slant curve if and only if the ratio of $\left(
\tau \pm 1\right) $ and $\kappa $ is constant, where $\tau $ is the geodesic
torsion and $\kappa $ is the geodesic curvature. Cabrerizo, Fernandez and
Gomez gave a nice geometric construction of an almost contact metric
structure compatible with an assigned metric on a $3$-dimensional oriented
Riemannian manifold in \cite{CFG-2009}. In the paper \cite{DIMN-2015}, Dru%
\c{t}\u{a}-Romaniuc, Inoguchi, Munteanu and Nistor studied the magnetic
trajectories of the contact magnetic field $F_{q}=q%
\Omega
$ on a Sasakian $(2n+1)$-manifold $(M^{2n+1},\varphi ,\xi ,\eta ,g)$, where $%
\Omega
$ is the fundamental $2$-form. The main objective of \cite{DIMN-2016} is the
study of trajectories for particles moving under the influence of a contact
magnetic curve in a cosymplectic manifold. The paper \cite{IM-2017} is
concerned with closed magnetic trajectories on 3-dimensional Berger spheres.
In \cite{JMN-2015}, the authors studied magnetic trajectories in an almost
contact metric manifold. They proved that normal magnetic curves are helices
of maximum order $5$. Moreover, in \cite{JM-2015}, Jleli and Munteanu worked
in the context of a para-Kaehler manifold, showing that spacelike and
timelike normal magnetic curves corresponding to the para-Kaehler $2$-forms
are circles. In \cite{MN-2012}, the authors gave a complete classification
of Killing magnetic curves with unit speed. Furthermore, in \cite{MN-2014},
the same authors proved that a normal magnetic curve on the Sasakian sphere $%
S^{2n+1}$ lies on a totally geodesic sphere $S^{3}$. They also considered
two particular magnetic fields on three-dimensional torus obtained from two
different contact forms on the Euclidean space $E^{3}$ and studied their
closed normal magnetic trajectories in their recent paper \cite{MN-2017}. In
\cite{OGY-2015}, the authors investigated some special curves in $3$%
-dimensional semi-Riemannian manifolds, such as $T$-magnetic curves, $N$%
-magnetic curves and $B$-magnetic curves, that are defined by means of their
Frenet elements. Calvaruso, Munteanu and Perrone provided a complete
classification of the magnetic trajectories of a Killing characteristic
vector field on an arbitrary normal paracontact metric manifold of dimension
3 in \cite{CMP-2015}. The present authors considered biharmonic Legendre
curves of $S$-space forms in \cite{OG-2014}. The second author studied
magnetic curves in the $3$-dimensional Heisenberg group in \cite{Ozgur-2017}%
. In \cite{Nak-1966}, Nakagawa introduced the notion of framed $f$%
-structures, which is a generalization of almost contact structures. Vanzura
studied almost $r$-structures in \cite{Vanzura-1972}. A differentiable
manifold with this structure is the same as a framed $f$-manifold as defined
by Nakagawa. On the other hand, Hasegawa, Okuyama and Abe defined a $p$th
Sasakian manifold and gave some typical examples in \cite{Hasegawa}.

Motivated by the above studies, in the present paper, we consider slant
normal magnetic curves in $(2n+s)$-dimensional $S$-manifolds. In Section \ref%
{Preliminaries}, we give brief information on $S$-manifolds and magnetic
curves. In Section \ref{magnetic}, we prove that $\gamma $ is a slant normal
magnetic curve in an $S$-manifold $(M^{2m+s},\varphi ,\xi _{\alpha },\eta
^{\alpha },g)$ if and only if it belongs to a list of slant $\varphi $%
-curves. This list consists of some specific geodesics, slant circles,
Legendre and slant helices of order $3$. Finally, in Section \ref%
{construction}, we construct slant normal magnetic curves in $%
\mathbb{R}
^{2n+s}(-3s)$ and give the parametric equations of these curves in two cases.

\section{Preliminaries\label{Preliminaries}}

In this section, we give brief information on $S$-manifolds and magnetic
curves. Let $\left( M^{2n+s},g\right) $ be a differentiable manifold, $%
\varphi $ a $(1,1)$-type tensor field, $\eta ^{\alpha }$ 1-forms, $\xi
_{\alpha }$ vector fields for $\alpha =1,...,s$, satisfying

\begin{equation}
\varphi ^{2}=-I+\overset{s}{\underset{\alpha =1}{\sum }}\eta ^{\alpha
}\otimes \xi _{\alpha },  \label{fisquare}
\end{equation}%
\begin{equation*}
\eta ^{\alpha }\left( \xi _{\beta }\right) =\delta _{\beta }^{\alpha },\text{
}\varphi \xi _{\alpha }=0,\text{ }\eta ^{\alpha }\left( \varphi X\right) =0,%
\text{ }\eta ^{\alpha }\left( X\right) =g\left( X,\xi _{\alpha }\right) ,
\end{equation*}%
\begin{equation}
g(\varphi X,\varphi Y)=g(X,Y)-\overset{s}{\underset{\alpha =1}{\sum }}\eta
^{\alpha }(X)\eta ^{\alpha }(Y),  \label{eq2}
\end{equation}%
\begin{equation*}
d\eta ^{\alpha }\left( X,Y\right) =-d\eta ^{\alpha }\left( Y,X\right)
=g(X,\varphi Y),
\end{equation*}%
where $X,Y\in TM$. Then $(\varphi ,\xi _{\alpha },\eta ^{\alpha },g)$ is
called \textit{framed }$\varphi $\textit{-structure} and $(M^{2n+s},\varphi
,\xi _{\alpha },\eta ^{\alpha },g)$ is called \textit{framed }$\varphi $%
\textit{-manifold \cite{Nak-1966}. }$(M^{2n+s},\varphi ,\xi _{\alpha },\eta
^{\alpha },g)$ is also called \textit{framed metric manifold} \cite{YK-1984}
or \textit{almost r-contact metric manifold }\cite{Vanzura-1972}. If the
Nijenhuis tensor of $\varphi $ is equal to $-2d\eta ^{\alpha }\otimes \xi
_{\alpha }$, then $(M^{2n+s},\varphi ,\xi _{\alpha },\eta ^{\alpha },g)$ is
called an $S$\textit{-manifold \cite{Blair-1970}}. For $s=1$, an $S$%
-structure becomes a Sasakian structure. For an $S$-structure, the following
properties are satisfied \cite{Blair-1970}:

\begin{equation}
(\nabla _{X}\varphi )Y=\underset{\alpha =1}{\overset{s}{\sum }}\left\{
g(\varphi X,\varphi Y)\xi _{\alpha }+\eta ^{\alpha }(Y)\varphi ^{2}X\right\}
,  \label{nablaf}
\end{equation}%
\begin{equation}
\nabla \xi _{\alpha }=-\varphi ,\text{ }\alpha \in \left\{ 1,...,s\right\} .
\label{nablaxi}
\end{equation}

Let $M^{2n+s}=(M^{2n+s},\varphi ,\xi _{\alpha },\eta ^{\alpha },g)$ be an $S$%
-manifold and $\Omega $ \textit{the fundamental }$2$\textit{-form} of $%
M^{2n+s}$ defined by
\begin{equation}
\Omega (X,Y)=g(X,\varphi Y),  \label{omega}
\end{equation}%
(see \cite{Nak-1966} and \cite{Vanzura-1972}). From the definition of framed
$\varphi $-structure, we have $\Omega =d\eta ^{\alpha }$. Hence, the
fundamental 2-form $\Omega $ on M$^{2n+s}$ is closed. The \textit{magnetic
field} $F_{q}$ on $M^{2n+s}$ can be defined by
\begin{equation*}
F_{q}(X,Y)=q\Omega (X,Y),
\end{equation*}%
where $X$ and $Y$ are vector fields on $M^{2n+s}$ and $q$ is a real
constant. $F_{q}$ is called the \textit{contact magnetic field with strength}
$q$ \cite{JMN-2015}. If $q=0$ then the magnetic curves are geodesics of $%
M^{2n+s}$. Because of this reason we shall consider $q\neq 0$ (see \cite%
{CFG-2009} and \cite{DIMN-2015}).

From (\ref{F}) and (\ref{omega}), the Lorentz force $\Phi $ associated to
the contact magnetic field $F_{q}$ can be written as
\begin{equation*}
\Phi _{q}=-q\varphi .
\end{equation*}%
So the Lorentz equation (\ref{Lorentz eq}) can be written as%
\begin{equation}
\nabla _{T}T=-q\varphi T,  \label{magneticcurve}
\end{equation}%
where $\gamma :I\subseteq R\rightarrow M^{2n+s}$ is a smooth unit-speed
curve and $T=\gamma ^{\prime }$ (see \cite{DIMN-2015} and \cite{JMN-2015}).

\section{Slant magnetic curves in $S$-manifolds\label{magnetic}}

Let $\left( M^{n},g\right) $ be a Riemannian manifold. A unit-speed curve $%
\gamma :I\rightarrow M$ is said to be \textit{a Frenet curve of osculating
order }$r$, if there exists positive functions $\kappa _{1},...,\kappa
_{r-1} $ on $I$ satisfying
\begin{eqnarray}
T &=&v_{1}=\gamma ^{\prime },  \notag \\
\nabla _{T}T &=&k_{1}v_{2},  \notag \\
\nabla _{T}v_{2} &=&-k_{1}T+k_{2}v_{3},  \label{Frenet} \\
&&...  \notag \\
\nabla _{T}v_{r} &=&-k_{r-1}v_{r-1},  \notag
\end{eqnarray}%
where $1\leq r\leq n$ and $T,v_{2},...,v_{r}$ are a $g$-orthonormal vector
fields along the curve. The positive functions $\kappa _{1},...,\kappa
_{r-1} $ are called\textit{\ curvature functions }and $\left\{
T,v_{2},...,v_{r}\right\} $ is called the \textit{Frenet frame field}. A
\textit{geodesic} is a Frenet curve of osculating order $r=1.$ A \textit{%
circle} is a Frenet curve of osculating order $r=2$ with a constant
curvature function $\kappa _{1}$. A \textit{helix of order} $r$ is a Frenet
curve of osculating order $r$ with constant curvature functions $\kappa
_{1},...,\kappa _{r-1}$. A helix of order $3$ is simply called a \textit{%
helix}.

Let $(M^{2m+s},\varphi ,\xi _{\alpha },\eta ^{\alpha },g)$ be an $S$%
-manifold. For a unit-speed curve $\gamma :I\rightarrow M$, if
\begin{equation*}
\eta ^{\alpha }(T)=0,
\end{equation*}%
for all $\alpha =1,...,s$, then $\gamma $ is called a \textit{Legendre curve
of} $M$ \cite{OG-2014}. More generally, if there exists a constant angle $%
\theta $ such that%
\begin{equation*}
\eta ^{\alpha }(T)=\cos \theta ,
\end{equation*}%
for all $\alpha =1,...,s$, then $\gamma $ is called a \textit{slant curve }%
and $\theta $ is called the \textit{contact angle of }$\gamma $, where $%
\left\vert \cos \theta \right\vert $ $\leq 1/\sqrt{s}$ \cite{GO-2018}.

Let $(M^{2m+s},\varphi ,\xi _{\alpha },\eta ^{\alpha },g)$ be an $S$%
-manifold. A Frenet curve of osculating order $r\geq 3$ is called a $\varphi
$\textit{-curve in }$M$ if its Frenet vector fields $T,v_{2},...,v_{r}$ span
a $\varphi $-invariant space. A $\varphi $-curve of osculating order $r$
with constant curvature functions $\kappa _{1},...,\kappa _{r-1}$ is called
a $\varphi $\textit{-helix of order} $r$. A curve of osculating order 2 is
called a $\varphi $\textit{-curve }if
\begin{equation*}
sp\left\{ T,v_{2},\overset{s}{\underset{\alpha =1}{\sum }}\xi _{\alpha
}\right\}
\end{equation*}%
is a $\varphi $-invariant space.

Throughout \ the paper, when we state "slant magnetic curve", we mean "slant
curves which satisfy equation (\ref{magneticcurve})". For magnetic curves, $%
\eta ^{\alpha }(T)=\cos \theta _{\alpha }$ does not have to be equal for all
$\alpha =1,...,s.$ By taking the curve as slant, we only study the equality
case of the slant angles $\theta _{\alpha }$ in the present paper. The
complete classification of magnetic curves in $S$-manifolds is still an open
problem.

Firstly, we state the following theorem:

\begin{theorem}
\label{theorem 3.1} Let $(M^{2m+s},\varphi ,\xi _{\alpha },\eta ^{\alpha
},g) $ be an $S$-manifold and consider the contact magnetic field $F_{q}$
for $q\neq 0$. Then $\gamma $ is a slant normal magnetic curve associated to
$F_{q}$ in $M^{2m+s}$ if and only if $\gamma $ belongs to the following list:

$a)$ geodesics obtained as integral curves of $\left( \pm \frac{1}{\sqrt{s}}%
\underset{\alpha =1}{\overset{s}{\sum }}\xi _{\alpha }\right) $;

$b)$ non-geodesic slant circles with the curvature $\kappa _{1}=\sqrt{q^{2}-s%
}$, having the contact angle $\theta =\arccos \left( \frac{1}{q}\right) $
and the Frenet frame field%
\begin{equation*}
\left\{ T,\frac{-sgn(q)\varphi T}{\sqrt{1-s\cos ^{2}\theta }}\right\} ,
\end{equation*}%
where $\left\vert q\right\vert >\sqrt{s}$ ;

$c)$ Legendre helices with curvatures $\kappa _{1}=\left\vert q\right\vert $
and $\kappa _{2}=\sqrt{s}$, having the Frenet frame field%
\begin{equation*}
\left\{ T,-sgn(q)\varphi T,\frac{-sgn(q)}{\sqrt{s}}\overset{s}{\underset{%
\alpha =1}{\sum }}\xi _{\alpha }\right\} ;
\end{equation*}%
i.e., a class of 1-dimensional integral submanifolds of the contact
distribution;

$d)$ slant helices with curvatures $\kappa _{1}=\left\vert q\right\vert
\sqrt{1-s\cos ^{2}\theta }$ and $\kappa _{2}=\sqrt{s}\left\vert 1-q\cos
\theta \right\vert $, having the Frenet frame field%
\begin{equation*}
\left\{ T,\frac{-sgn(q)\varphi T}{\sqrt{1-s\cos ^{2}\theta }},\frac{%
-\varepsilon sgn(q)}{\sqrt{s}\sqrt{1-s\cos ^{2}\theta }}\left( -s\cos \theta
T+\overset{s}{\underset{\alpha =1}{\sum }}\xi _{\alpha }\right) \right\} ,
\end{equation*}%
where $\theta \neq \frac{\pi }{2}$ is the contact angle satisfying $%
\left\vert \cos \theta \right\vert <\frac{1}{\sqrt{s}}$ and $\varepsilon
=sgn(1-q\cos \theta )$.
\end{theorem}

\begin{proof}
Let $\gamma $ be a normal magnetic curve. If the magnetic curve is a
geodesic, then
\begin{equation*}
\nabla _{T}T=0=-q\varphi T
\end{equation*}%
gives us%
\begin{equation*}
T\in sp\left\{ \xi _{1},...,\xi _{s}\right\} .
\end{equation*}%
If $\gamma $ is slant, then we can write%
\begin{equation*}
T=\cos \theta \underset{\alpha =1}{\overset{s}{\sum }}\xi _{\alpha }.
\end{equation*}%
Since $\gamma $ is unit speed, we have $\cos \theta =\pm \frac{1}{\sqrt{s}}$%
. So the proof of a) is complete.

From now on, we suppose that $\gamma $ is a non-geodesic Frenet curve of
osculating order $r>1$. Let us choose an $\alpha \in \left\{ 1,...,s\right\}
$. Applying $\xi _{\alpha }$ to $\nabla _{T}T=-q\varphi T$, we obtain%
\begin{equation}
0=g\left( -q\varphi T,\xi _{\alpha }\right) =g\left( \nabla _{T}T,\xi
_{\alpha }\right) =\frac{d}{dt}g\left( T,\xi _{\alpha }\right) -g\left(
T,\nabla _{T}\xi _{\alpha }\right) .  \label{equaaa1}
\end{equation}%
From (\ref{nablaxi}), we also have%
\begin{equation}
\nabla _{T}\xi _{\alpha }=-\varphi T.  \label{nablaksi}
\end{equation}%
Using equations (\ref{equaaa1}) and (\ref{nablaksi}), we find
\begin{equation*}
\frac{d}{dt}g\left( T,\xi _{\alpha }\right) =0,
\end{equation*}%
that is,%
\begin{equation*}
\eta ^{\alpha }\left( T\right) =\cos \theta _{\alpha }=constant.
\end{equation*}%
Let us assume $\theta _{\alpha }=\theta $ for all $\alpha =1,...,s$, i.e., $%
\gamma $ is slant. So, we have
\begin{equation}
\eta ^{\alpha }\left( T\right) =\cos \theta .  \label{etaalfaT3.4}
\end{equation}%
Equations (\ref{magneticcurve}) and (\ref{Frenet}) give us%
\begin{equation}
\nabla _{T}T=\kappa _{1}v_{2}=-q\varphi T.  \label{eq3.5}
\end{equation}%
Then we get%
\begin{equation}
\kappa _{1}=\left\vert q\right\vert \left\Vert \varphi T\right\Vert
=\left\vert q\right\vert \sqrt{1-s\cos ^{2}\theta }.  \label{eq3.6}
\end{equation}%
If we write (\ref{eq3.6}) in (\ref{eq3.5}), we find%
\begin{equation*}
-q\varphi T=\kappa _{1}v_{2}=\left\vert q\right\vert \sqrt{1-s\cos
^{2}\theta }v_{2},
\end{equation*}%
which gives us%
\begin{equation}
\varphi T=-\frac{\left\vert q\right\vert }{q}\sqrt{1-s\cos ^{2}\theta }%
v_{2}=-sgn(q)\sqrt{1-s\cos ^{2}\theta }v_{2}.  \label{eq3.7}
\end{equation}%
If $\kappa _{2}=0$, then the magnetic curve is a Frenet curve of osculating
order $r=2$. Since $\kappa _{1}$ is a constant, $\gamma $ is a circle. From (%
\ref{eq3.7}), we have%
\begin{equation*}
\eta ^{\alpha }\left( \varphi T\right) =0=-sgn(q)\sqrt{1-s\cos ^{2}\theta }%
\eta ^{\alpha }\left( v_{2}\right) \text{,}
\end{equation*}%
that is,%
\begin{equation*}
\eta ^{\alpha }\left( v_{2}\right) =0.
\end{equation*}%
If we differentiate the last equation along the curve $\gamma $, we obtain%
\begin{equation*}
\nabla _{T}\eta ^{\alpha }\left( v_{2}\right) =0=g\left( \nabla
_{T}v_{2},\xi _{\alpha }\right) +g\left( v_{2},\nabla _{T}\xi _{\alpha
}\right) .
\end{equation*}%
So, we calculate%
\begin{equation*}
g\left( -\kappa _{1}T,\xi _{\alpha }\right) +g\left( v_{2},sgn(q)\sqrt{%
1-s\cos ^{2}\theta }v_{2}\right) =0.
\end{equation*}%
Since $r=2$, we find%
\begin{equation*}
-\kappa _{1}\cos \theta +sgn(q)\sqrt{1-s\cos ^{2}\theta }=0.
\end{equation*}%
Using equation (\ref{eq3.6}) in the last equation, it is easy to see that%
\begin{equation*}
\left\vert q\right\vert \sqrt{1-s\cos ^{2}\theta }\left( -\cos \theta +\frac{%
1}{q}\right) =0.
\end{equation*}%
Since $\gamma $ is non-geodesic, we have%
\begin{equation*}
\cos \theta =\frac{1}{q}.
\end{equation*}%
Then equation (\ref{eq3.6}) becomes%
\begin{equation*}
\kappa _{1}=\left\vert q\right\vert \sqrt{1-s\cos ^{2}\theta }=\sqrt{q^{2}-s}%
,
\end{equation*}%
where $\left\vert q\right\vert >\sqrt{s}$. So the proof of b) is complete.

Let $\kappa _{2}\neq 0$. From (\ref{fisquare}) and (\ref{etaalfaT3.4}), we
find%
\begin{equation}
\varphi ^{2}T=-T+\cos \theta \underset{\alpha =1}{\overset{s}{\sum }}\xi
_{\alpha }.  \label{fikareT}
\end{equation}%
Using (\ref{nablaf}) and (\ref{etaalfaT3.4}), we have%
\begin{equation}
(\nabla _{T}\varphi )T=-s\cos \theta T+\underset{\alpha =1}{\overset{s}{\sum
}}\xi _{\alpha },  \label{eqaa}
\end{equation}%
which gives us%
\begin{eqnarray}
\nabla _{T}\varphi T &=&(\nabla _{T}\varphi )T+\varphi \nabla _{T}T  \notag
\\
&=&-s\cos \theta T+\underset{\alpha =1}{\overset{s}{\sum }}\xi _{\alpha
}+\varphi (-q\varphi T)  \notag \\
&=&-s\cos \theta T+\underset{\alpha =1}{\overset{s}{\sum }}\xi _{\alpha
}-q\left( -T+\cos \theta \underset{\alpha =1}{\overset{s}{\sum }}\xi
_{\alpha }\right) .  \label{star1}
\end{eqnarray}%
Differentiating (\ref{eq3.7}), we also find%
\begin{equation}
\nabla _{T}\varphi T=-sqn(q)\sqrt{1-s\cos ^{2}\theta }\left( -\kappa
_{1}T+\kappa _{2}v_{3}\right) .  \label{star2}
\end{equation}%
By the use of (\ref{eq3.6}), (\ref{star1}) and (\ref{star2}), after some
calculations, we obtain%
\begin{equation}
(1-q\cos \theta )\left( -s\cos \theta T+\underset{\alpha =1}{\overset{s}{%
\sum }}\xi _{\alpha }\right) =-sgn(q)\sqrt{1-s\cos ^{2}\theta }\kappa
_{2}v_{3}.  \label{eqbb}
\end{equation}%
If we find the norm of both sides in (\ref{eqbb}), we get%
\begin{equation}
\kappa _{2}=\sqrt{s}\left\vert 1-q\cos \theta \right\vert .  \label{kappa2}
\end{equation}%
Let us denote $\varepsilon =sgn(1-q\cos \theta )$. If we write (\ref{kappa2}%
) in (\ref{eqbb}), we obtain%
\begin{equation}
\underset{\alpha =1}{\overset{s}{\sum }}\xi _{\alpha }=s\cos \theta
T-\varepsilon sgn(q)\sqrt{s}\sqrt{1-s\cos ^{2}\theta }v_{3}.  \label{eq3.9}
\end{equation}%
Applying $\varphi $ to (\ref{eq3.9}), we find%
\begin{equation*}
\varphi v_{3}=-\varepsilon \sqrt{s}\cos \theta v_{2}.
\end{equation*}%
If we apply $\varphi $ to (\ref{eq3.7}) and then use equations (\ref%
{etaalfaT3.4}) and (\ref{eq3.9}) together, we have%
\begin{equation}
\varphi v_{2}=sgn(q)\sqrt{1-s\cos ^{2}\theta }T+\varepsilon \cos \theta
\sqrt{s}v_{3}.  \label{eq3.10}
\end{equation}%
Let us choose a $\beta \in \left\{ 1,...,s\right\} .$ From (\ref{eq3.10}),
we calculate%
\begin{equation*}
\eta ^{\beta }(v_{3})=-\varepsilon sgn(q)\frac{\sqrt{1-s\cos ^{2}\theta }}{%
\sqrt{s}}.
\end{equation*}%
If we differentiate (\ref{eq3.9}) along the curve $\gamma $, we get%
\begin{equation*}
\underset{\alpha =1}{\overset{s}{\sum }}\nabla _{T}\xi _{\alpha }=s\cos
\theta \nabla _{T}T-\varepsilon sgn(q)\sqrt{s}\sqrt{1-s\cos ^{2}\theta }%
\nabla _{T}v_{3},
\end{equation*}%
which gives us%
\begin{equation*}
-s\left( 1-q\cos \theta \right) \varphi T=-\varepsilon sgn(q)\sqrt{s}\sqrt{%
1-s\cos ^{2}\theta }\left( -\kappa _{2}v_{2}+\kappa _{3}v_{4}\right) .
\end{equation*}%
Since $\varphi T\parallel v_{2}$, we find $\kappa _{3}=0$. This proves d) of
the theorem.

Let us examine Legendre case separately, that is, $\theta =\frac{\pi }{2}$.
Then we have $\varepsilon =1$, $\kappa _{1}=\left\vert q\right\vert $, $%
\kappa _{2}=\sqrt{s}$, $\kappa _{3}=0$ and equation (\ref{eq3.9}) gives us%
\begin{equation*}
v_{3}=\frac{-sgn(q)}{\sqrt{s}}\underset{\alpha =1}{\overset{s}{\sum }}\xi
_{\alpha }.
\end{equation*}%
This completes the proof of c).

Conversely, let $\gamma $ satisfy one of $a)$, $b)$, $c)$ or $d)$. Using the
Frenet frame field and Frenet equations, it is straightforward to show that $%
\nabla _{T}T=-q\varphi T$, i.e., $\gamma $ is a slant normal magnetic curve.
\end{proof}

The above theorem is a generalization of Theorem 3.1 of \cite{DIMN-2015} (by
Simona Luiza Druta-Romaniuc et al.) for $S$-manifolds. If we choose $s=1$,
since an $S$-manifold becomes a Sasakian manifold, we find their results.

\textbf{Remark.} The order of a slant magnetic curve in an $S$-manifold is
still $r\leq 3$, as in the case of a magnetic curve of a Sasakian manifold,
which was considered in \cite{DIMN-2015}.

Now, let us remove the slant condition from the hypothesis and show that the
osculating order is still $r\leq 3.$

\begin{theorem}
\label{order} Let $(M^{2m+s},\varphi ,\xi _{\alpha },\eta ^{\alpha },g)$ be
an $S$-manifold and consider the contact magnetic field $F_{q}$ for $q\neq 0$%
. If $\gamma $ is a normal magnetic curve associated to $F_{q}$ in $M^{2m+s}$%
, then the osculating order $r\leq 3$.
\end{theorem}

\begin{proof}
Let $\gamma $ be a normal magnetic curve. Then, the Lorentz equation (\ref%
{magneticcurve}) gives us%
\begin{equation*}
\eta ^{\alpha }(T)=\cos \theta _{\alpha }\text{, }\alpha =1,...,s.
\end{equation*}%
If we differentiate this equation along the curve, we have
\begin{equation*}
\eta ^{\alpha }(E_{2})=0
\end{equation*}%
for all $\alpha =1,...,s.$ From the Frenet equations (\ref{Frenet}), we
obtain%
\begin{equation*}
-q\varphi T=\kappa _{1}v_{2}.
\end{equation*}%
From the definition of framed $\varphi $-structure, we calculate%
\begin{equation*}
g(\varphi T,\varphi T)=1-A,
\end{equation*}%
where we denote
\begin{equation*}
A=\overset{s}{\underset{\alpha =1}{\sum }}\cos ^{2}\theta _{\alpha }.
\end{equation*}%
Then, we have%
\begin{equation*}
\left\Vert \varphi T\right\Vert =\sqrt{1-A}
\end{equation*}%
and%
\begin{equation*}
\kappa _{1}=\left\vert q\right\vert \sqrt{1-A}.
\end{equation*}%
Thus, $\varphi T$ can be rewritten as%
\begin{equation}
\varphi T=-sgn(q)\sqrt{1-A}v_{2}.  \label{eq}
\end{equation}%
Again, from the definition of framed $\varphi $-structure, we have%
\begin{equation*}
\varphi ^{2}T=-T+V,
\end{equation*}%
where we denote%
\begin{equation*}
V=\overset{s}{\underset{\alpha =1}{\sum }}\cos \theta _{\alpha }\xi _{\alpha
}.
\end{equation*}%
After some calculations, we get%
\begin{eqnarray*}
\nabla _{T}\varphi T &=&(q-B)T+(1-A)\overset{s}{\underset{\alpha =1}{\sum }}%
\xi _{\alpha } \\
&&+(-q+B)V,
\end{eqnarray*}%
which corresponds to equation (\ref{star1}). Here, we denote%
\begin{equation*}
B=\overset{s}{\underset{\alpha =1}{\sum }}\cos \theta _{\alpha }.
\end{equation*}%
From equation (\ref{eq}), we also find%
\begin{equation*}
\nabla _{T}\varphi T=-sqn(q)\sqrt{1-A}(-\kappa _{1}T+\kappa _{2}v_{3}),
\end{equation*}%
which corresponds to equation (\ref{star2}). In this last equation, we can
replace $\kappa _{1}=\left\vert q\right\vert \sqrt{1-A}$. Finally, we have%
\begin{eqnarray}
-sqn(q)\sqrt{1-A}\kappa _{2}v_{3} &=&\left( 1-A\right) \overset{s}{\underset{%
\alpha =1}{\sum }}\xi _{\alpha }+\left( -q+B\right) V  \label{eqx} \\
&&+\left( qA-B\right) T.  \notag
\end{eqnarray}%
So, if we denote the norm of the right hand side of equation (\ref{eqx}) by $%
C$, we find%
\begin{equation*}
C=\sqrt{(1-A)(Aq^{2}-As+B^{2}-2Bq+s)},
\end{equation*}%
which is a constant. Hence, we obtain%
\begin{equation*}
\kappa _{2}=\frac{C}{\sqrt{1-A}}=\sqrt{Aq^{2}-As+B^{2}-2Bq+s}=\text{constant}%
.
\end{equation*}%
From equation (\ref{eqx}), we also have $v_{3}\in span\left\{ T,\xi
_{1},...,\xi _{s}\right\} .$ The angles between $v_{3}$ and $T,\xi
_{1},...,\xi _{s}$ are all constants since all the coefficients in equation (%
\ref{eqx}) are constants. Then, we can write%
\begin{equation}
v_{3}=c_{0}T+c_{1}\xi _{1}+...+c_{s}\xi _{s}  \label{eq3}
\end{equation}%
for some constants $c_{0},...,c_{s}.$ If we differentiate equation (\ref{eq3}%
), we get%
\begin{equation*}
-\kappa _{2}v_{2}+\kappa _{3}v_{4}=c_{0}\kappa _{1}v_{2}-c_{1}\varphi
T-...-c_{s}\varphi T.
\end{equation*}%
Since $\varphi T$ is parallel to $v_{2}$, if we take the inner product of
the last equation with $v_{4}$, we find $\kappa _{3}=0$. This proves the
theorem.
\end{proof}

In particular, if $\gamma $ is slant, i.e. $\theta _{\alpha }=\theta $ for
all $\alpha =1,...,s,$ then we obtain the following corollary:

\begin{corollary}
If $\theta _{\alpha }=\theta ,$ for all $\alpha =1,...,s,$ then%
\begin{equation*}
A=s\cos ^{2}\theta ,\text{ }B=s\cos \theta ,\text{ }V=\cos \theta \overset{s}%
{\underset{\alpha =1}{\sum }}\xi _{\alpha },
\end{equation*}%
\begin{equation*}
C=\sqrt{(1-s\cos ^{2}\theta )s(1-q\cos \theta )^{2}}
\end{equation*}%
and $\kappa _{2}=\sqrt{s}\left\vert 1-q\cos \theta \right\vert .$
\end{corollary}

Now, let us state the following proposition:

\begin{proposition}
\label{PROPSLANT}Let $\gamma $ be a slant $\varphi $-helix of order 3 in an $%
S$-manifold $(M^{2m+s},\varphi ,\xi _{\alpha },\eta ^{\alpha },g)$ with
contact angle $\theta $. Then
\begin{equation}
\underset{\alpha =1}{\overset{s}{\sum }}\xi _{\alpha }=s\cos \theta T+\rho
v_{3},  \label{sumksialfa}
\end{equation}%
where $\rho =g\left( v_{3},\underset{\alpha =1}{\overset{s}{\sum }}\xi
_{\alpha }\right) =s\eta ^{\alpha }\left( v_{3}\right) $ is a real constant
such that $\rho ^{2}=s-s^{2}\cos ^{2}\theta $. Hence, $\gamma $ has the
Frenet frame field%
\begin{equation*}
\left\{ T,\frac{\pm \varphi T}{\sqrt{1-s\cos ^{2}\theta }},\frac{\pm 1}{%
\sqrt{s}\sqrt{1-s\cos ^{2}\theta }}\left( -s\cos \theta T+\overset{s}{%
\underset{\alpha =1}{\sum }}\xi _{\alpha }\right) \right\} .
\end{equation*}
\end{proposition}

\begin{proof}
From the assumption, the Frenet frame field $\left\{ T,v_{2},v_{3}\right\} $
is $\varphi $-invariant and%
\begin{equation*}
\eta ^{\alpha }(T)=\cos \theta \text{.}
\end{equation*}%
Differentiating the last equation along the curve, it is easy to see that%
\begin{equation}
\eta ^{\alpha }(v_{2})=0.  \label{etav2}
\end{equation}%
If we differentiate once again, we have%
\begin{equation}
g\left( \varphi T,v_{2}\right) =-\kappa _{1}\cos \theta +\kappa _{2}\eta
^{\alpha }(v_{3}),  \label{gfitev2}
\end{equation}%
which means the value of $\eta ^{\alpha }(v_{3})$ does not depend on $\alpha
.$

Firstly, let us assume that $\theta \neq \frac{\pi }{2}$. Since the\ space
spanned by the Frenet frame field is $\varphi $-invariant, then $\varphi
^{2}T$ is in the set. Using (\ref{fikareT}) and (\ref{etav2}), we can write
\begin{equation*}
\underset{\alpha =1}{\overset{s}{\sum }}\xi _{\alpha }\in sp\left\{
T,v_{3}\right\} ,
\end{equation*}%
that is,%
\begin{equation}
\underset{\alpha =1}{\overset{s}{\sum }}\xi _{\alpha }=s\cos \theta T+\rho
v_{3}.  \label{sumksi}
\end{equation}%
If we take the norm of both sides, we find $\rho ^{2}=s-s^{2}\cos ^{2}\theta
$. Since the value of $\eta ^{\alpha }(v_{3})$ does not depend on $\alpha $,
we obtain%
\begin{equation*}
\rho =g\left( v_{3},\underset{\alpha =1}{\overset{s}{\sum }}\xi _{\alpha
}\right) =s\eta ^{\alpha }\left( v_{3}\right) .
\end{equation*}

If we apply $\varphi T$ to (\ref{sumksi}), we get $g(\varphi T,v_{3})=0$.
Since $\varphi T\perp T$ , $\varphi T\perp v_{3}$ and $sp\left\{
T,v_{2},v_{3}\right\} $ is $\varphi $-invariant, we have $\varphi T\parallel
v_{2}$. As a result, we find%
\begin{equation*}
v_{2}=\frac{\pm \varphi T}{\left\Vert \varphi T\right\Vert }=\frac{\pm
\varphi T}{\sqrt{1-s\cos ^{2}\theta }}.
\end{equation*}

Now let us consider the Legendre case, i.e., $\theta =\frac{\pi }{2}$. From (%
\ref{gfitev2}), we find%
\begin{equation}
\nabla _{T}g\left( \varphi T,v_{2}\right) =-\kappa _{2}g(\varphi T,v_{3}).
\label{starr2}
\end{equation}%
Using (\ref{fisquare}) and (\ref{nablaf}), we calculate%
\begin{eqnarray}
\nabla _{T}\varphi T &=&(\nabla _{T}\varphi )T+\varphi \nabla _{T}T
\label{nablafiTson} \\
&=&\underset{\alpha =1}{\overset{s}{\sum }}\xi _{\alpha }+\kappa _{1}\varphi
v_{2}.  \notag
\end{eqnarray}%
Using the last equation, we obtain%
\begin{eqnarray}
\nabla _{T}g\left( \varphi T,v_{2}\right) &=&g\left( \nabla _{T}\varphi
T,v_{2}\right) +g(\varphi T,\nabla _{T}v_{2})  \label{starr1} \\
&=&\kappa _{2}g\left( \varphi T,v_{3}\right) .  \notag
\end{eqnarray}%
Equations (\ref{starr2}) and (\ref{starr1}) give us $g\left( \varphi
T,v_{3}\right) =0$, that is, $\varphi T\parallel v_{2}$. Thus, we have $%
\varphi T=\pm v_{2}$. Consequently, the Frenet frame field becomes $\left\{
T,\pm \varphi T,v_{3}\right\} $. Now, we must show that $v_{3}$ is parallel
to $\underset{\alpha =1}{\overset{s}{\sum }}\xi _{\alpha }$. Since the space
spanned by the Frenet frame field is $\varphi $-invariant, from orthonormal
expansion, we can write%
\begin{equation*}
\varphi v_{3}=g(\varphi v_{3},T)T\pm g\left( \varphi v_{3},\pm \varphi
T\right) \varphi T+g\left( \varphi v_{3},v_{3}\right) v_{3},
\end{equation*}%
which reduces to%
\begin{equation}
\varphi v_{3}=g(\varphi v_{3},T)T.  \label{fiv3}
\end{equation}%
If we apply $\varphi $ to equation (\ref{fiv3}) and use (\ref{fisquare}), we
find%
\begin{equation}
-v_{3}+\underset{\alpha =1}{\overset{s}{\sum }}\eta ^{\alpha }(v_{3})\xi
_{\alpha }=g(\varphi v_{3},T)\varphi T.  \label{equ1}
\end{equation}%
Applying $\varphi T$ to (\ref{equ1}) and using the Frenet frame field, we
have $g(\varphi v_{3},T)=0$. As a result, we get $\varphi v_{3}=0$ and
equation (\ref{equ1}) becomes%
\begin{equation*}
-v_{3}+\underset{\alpha =1}{\overset{s}{\sum }}\eta ^{\alpha }(v_{3})\xi
_{\alpha }=0.
\end{equation*}%
We have already shown that the value of $\eta ^{\alpha }(v_{3})$ does not
depend on $\alpha $; so, we can write
\begin{equation}
v_{3}=\eta ^{\alpha }(v_{3})\underset{\alpha =1}{\overset{s}{\sum }}\xi
_{\alpha }.  \label{v3}
\end{equation}%
Since $v_{3}$ and $\xi _{\alpha }$ are unit for all $\alpha =1,...,s,$ we
find $\eta ^{\alpha }(v_{3})=\frac{\pm 1}{\sqrt{s}}.$ Finally, for $\theta =%
\frac{\pi }{2},$ we have $\rho ^{2}=s$ and $\underset{\alpha =1}{\overset{s}{%
\sum }}\xi _{\alpha }=\rho v_{3}$, which completes the proof.
\end{proof}

\begin{corollary}
\label{Prop3.1}Let $\gamma $ be a Legendre $\varphi $-helix of order 3 in an
$S$-manifold $(M^{2m+s},\varphi ,\xi _{\alpha },\eta ^{\alpha },g)$. Then $%
\kappa _{2}=\sqrt{s}$, $v_{2}=\pm \varphi T$ and $v_{3}=\pm \frac{1}{\sqrt{s}%
}\underset{\alpha =1}{\overset{s}{\sum }}\xi _{\alpha }$.
\end{corollary}

\begin{proof}
From equation (\ref{v3}), we already have
\begin{equation*}
v_{3}=\pm \frac{1}{\sqrt{s}}\underset{\alpha =1}{\overset{s}{\sum }}\xi
_{\alpha }.
\end{equation*}%
If we differentiate this equation and use (\ref{Frenet}), we obtain%
\begin{equation}
-\kappa _{2}v_{2}=\pm \frac{1}{\sqrt{s}}\underset{\alpha =1}{\overset{s}{%
\sum }}\nabla _{T}\xi _{\alpha }.  \label{eq11}
\end{equation}%
Using equations (\ref{nablaxi}) and (\ref{eq11}), we find that $\kappa _{2}=%
\sqrt{s}$ and $v_{2}=\pm \varphi T$.
\end{proof}

Finally, we can give the following theorem:

\begin{theorem}
\label{Theorem32}Let $\gamma $ be a slant $\varphi $-helix of order $r\leq 3$
on an $S$-manifold $(M^{2m+s},\varphi ,\xi _{\alpha },\eta ^{\alpha },g)$.
Let $\theta $ denote the contact angle of $\gamma $. Then we have

i. If $\cos \theta =\pm \frac{1}{\sqrt{s}}$, then $\gamma $ is an integral
curve of \ $\pm \frac{1}{\sqrt{s}}\underset{\alpha =1}{\overset{s}{\sum }}%
\xi _{\alpha }$, hence it is a normal magnetic curve for $F_{q}$ with an
arbitrary $q$.

ii. If $\cos \theta =0$ and $\kappa _{1}\neq 0$ (i.e. $\gamma $ is a
non-geodesic Legendre curve), then $\gamma $ is a magnetic curve for $F_{\pm
\kappa _{1}}$.

iii. If $\cos \theta =\frac{\varepsilon }{\sqrt{\kappa _{1}^{2}+s}}$, then $%
\gamma $ is a magnetic curve for $F_{\varepsilon \sqrt{\kappa _{1}^{2}+s}}$,
where $\varepsilon =-sgn(g(\varphi T,v_{2}))$. In this case, $\gamma $ is a
slant $\varphi $-circle.

iv. If $\cos \theta =\frac{\varepsilon \sqrt{s}\pm \kappa _{2}}{\sqrt{s}%
\sqrt{\kappa _{1}^{2}+\left( \varepsilon \sqrt{s}\pm \kappa _{2}\right) ^{2}}%
}$, then $\gamma $ is a magnetic curve for $F_{\varepsilon \sqrt{\kappa
_{1}^{2}+\left( \varepsilon \sqrt{s}\pm \kappa _{2}\right) ^{2}}}$, where $%
\varepsilon =-sgn(g(\varphi T,v_{2}))$ and the sign $\pm $ corresponds to
the sign of $\eta ^{\alpha }(v_{3})$.

v. Except the above cases, $\gamma $ is not a magnetic curve for any $F_{q}$.
\end{theorem}

\begin{proof}
Let $\cos \theta =\pm \frac{1}{\sqrt{s}}$. Then we have%
\begin{equation*}
T=\pm \frac{1}{\sqrt{s}}\underset{\alpha =1}{\overset{s}{\sum }}\xi _{\alpha
},
\end{equation*}%
which gives us $\nabla _{T}T=0$. We also have $\varphi T=0$. So $\gamma $
satisfies $\nabla _{T}T=-q\varphi T$ for any $q$, which proves \textit{i}.

Now let $\cos \theta =0$ and $\kappa _{1}\neq 0$. Using Corollary \ref%
{Prop3.1}, we have%
\begin{equation*}
\nabla _{T}T=\kappa _{1}\left( \pm \varphi T\right) =-q\varphi T,
\end{equation*}%
which gives us $q=\pm \kappa _{1}$. This completes the proof of \textit{ii}.

From Proposition \ref{PROPSLANT}, we have the Frenet frame field
\begin{equation*}
\left\{ T,\frac{\pm \varphi T}{\sqrt{1-s\cos ^{2}\theta }},\frac{\pm 1}{%
\sqrt{s}\sqrt{1-s\cos ^{2}\theta }}\left( -s\cos \theta T+\overset{s}{%
\underset{\alpha =1}{\sum }}\xi _{\alpha }\right) \right\}
\end{equation*}%
when $r=3$ and
\begin{equation*}
\left\{ T,\frac{\pm \varphi T}{\sqrt{1-s\cos ^{2}\theta }}\right\}
\end{equation*}%
when $r=2$. If we differentiate $v_{2}$ along the curve, after some
calculations, in both cases, we find%
\begin{equation}
\left( 1\pm \frac{\kappa _{1}\cos \theta }{\sqrt{1-s\cos ^{2}\theta }}%
\right) \left( -s\cos \theta T+\overset{s}{\underset{\alpha =1}{\sum }}\xi
_{\alpha }\right) =\pm \kappa _{2}\sqrt{1-s\cos ^{2}\theta }v_{3},
\label{important}
\end{equation}%
(taking $\kappa _{2}=0$, when $r=2$).

Next, let us assume $\cos \theta =\frac{\varepsilon }{\sqrt{\kappa _{1}^{2}+s%
}}$, where we denote $\varepsilon =-sgn(g(\varphi T,v_{2}))$. Then the left
side of equation (\ref{important}) vanishes. Thus we get $\kappa _{2}=0$.
From the assumption, we also have $\kappa _{1}=constant$, that is, $\gamma $
is a slant $\varphi $-circle. Using the Frenet frame field, we find $\nabla
_{T}T=-q\varphi T=\kappa _{1}v_{2}$, where $q=\varepsilon \sqrt{\kappa
_{1}^{2}+s}$. So, we have just completed the proof of \textit{iii}.

Finally, let us assume $\cos \theta =\frac{\varepsilon \sqrt{s}\pm \kappa
_{2}}{\sqrt{s}\sqrt{\kappa _{1}^{2}+\left( \varepsilon \sqrt{s}\pm \kappa
_{2}\right) ^{2}}}$, where $\varepsilon =-sgn(g(\varphi T,v_{2}))$ and the
sign $\pm $ corresponds to the sign of $\eta ^{\alpha }(v_{3})$. In this
case, let us take $\kappa _{2}\neq 0$, since we have already investigated
order $r=2$. Using the Frenet frame field, after some calculations, we
obtain $\nabla _{T}T=-q\varphi T=\kappa _{1}v_{2}$, where $q=\varepsilon
\sqrt{\kappa _{1}^{2}+\left( \varepsilon \sqrt{s}\pm \kappa _{2}\right) ^{2}}
$. Hence, the proof of \textit{iv} is complete.

Since we have considered all cases, we can state that there exist no other
slant magnetic $\varphi $-helices in $M$.
\end{proof}

From the proof of Theorem \ref{theorem 3.1}, we can give the following
proposition:

\begin{proposition}
\label{prop2}Let $(M^{2m+s},\varphi ,\xi _{\alpha },\eta ^{\alpha },g)$ be
an $S$-manifold. There exist no non-geodesic slant $\varphi $-circles as
magnetic curves corresponding to $F_{q}$ for $0<\left\vert q\right\vert \leq
\sqrt{s}$.
\end{proposition}

Theorem \ref{Theorem32} and Proposition \ref{prop2} generalize Theorem 3.2
and Proposition 3.2 in \cite{DIMN-2015} to $S$-manifolds, respectively.
Under the condition $s=1$, we obtain their results.

\section{\protect\bigskip Construction of slant normal magnetic curves in $%
\mathbb{R}
^{2n+s}(-3s)$\label{construction}}

In this section, we find parametric equations of slant normal magnetic
curves in $%
\mathbb{R}
^{2n+s}(-3s)$. As a start, we recall structures defined on this $S$%
-manifold. Let us take $M=%
\mathbb{R}
^{2n+s}$ with coordinate functions $\left\{
x_{1},...x_{n},y_{1},...,y_{n},z_{1},...,z_{s}\right\} $ and define%
\begin{equation*}
\xi _{\alpha }=2\frac{\partial }{\partial z_{\alpha }},\text{ }\alpha
=1,...,s,
\end{equation*}%
\begin{equation*}
\eta ^{\alpha }=\frac{1}{2}\left( dz_{\alpha }-\overset{n}{\underset{i=1}{%
\sum }}y_{i}dx_{i}\right) ,\text{ }\alpha =1,...,s,
\end{equation*}%
\begin{equation*}
\varphi X=\overset{n}{\underset{i=1}{\sum }}Y_{i}\frac{\partial }{\partial
x_{i}}-\overset{n}{\underset{i=1}{\sum }}X_{i}\frac{\partial }{\partial y_{i}%
}+\left( \overset{n}{\underset{i=1}{\sum }}Y_{i}y_{i}\right) \left( \overset{%
s}{\underset{\alpha =1}{\sum }}\frac{\partial }{\partial z_{\alpha }}\right)
,
\end{equation*}%
\begin{equation*}
g=\overset{s}{\underset{\alpha =1}{\sum }}\eta ^{\alpha }\otimes \eta
^{\alpha }+\frac{1}{4}\overset{n}{\underset{i=1}{\sum }}\left( dx_{i}\otimes
dx_{i}+dy_{i}\otimes dy_{i}\right) ,
\end{equation*}%
where%
\begin{equation*}
X=\overset{n}{\underset{i=1}{\sum }}\left( X_{i}\frac{\partial }{\partial
x_{i}}+Y_{i}\frac{\partial }{\partial y_{i}}\right) +\overset{s}{\underset{%
\alpha =1}{\sum }}\left( Z_{\alpha }\frac{\partial }{\partial z_{\alpha }}%
\right) \in \chi (M).
\end{equation*}

\bigskip It is well-known that $\left(
\mathbb{R}
^{2n+s},\varphi ,\xi _{\alpha },\eta ^{\alpha },g\right) $ is an $S$-space
form with constant $\varphi $-sectional curvature $-3s.$ Hence it is denoted
by $%
\mathbb{R}
^{2n+s}(-3s)$ \cite{Hasegawa}. The following vector fields
\begin{equation*}
X_{i}=2\frac{\partial }{\partial y_{i}},\text{ }X_{n+i}=\varphi X_{i}=2(%
\frac{\partial }{\partial x_{i}}+y_{i}\overset{s}{\underset{\alpha =1}{\sum }%
}\frac{\partial }{\partial z_{\alpha }}),\text{ }\xi _{\alpha }=2\frac{%
\partial }{\partial z_{\alpha }}
\end{equation*}%
form a $g$-orthonormal basis and the Levi-Civita connection is%
\begin{equation*}
\nabla _{X_{i}}X_{j}=\nabla _{X_{n+i}}X_{n+j}=0,\nabla
_{X_{i}}X_{n+j}=\delta _{ij}\overset{s}{\underset{\alpha =1}{\sum }}\xi
_{\alpha },\nabla _{X_{n+i}}X_{j}=-\delta _{ij}\overset{s}{\underset{\alpha
=1}{\sum }}\xi _{\alpha },
\end{equation*}%
\begin{equation*}
\nabla _{X_{i}}\xi _{\alpha }=\nabla _{\xi _{\alpha }}X_{i}=-X_{n+i},\nabla
_{X_{n+i}}\xi _{\alpha }=\nabla _{\xi _{\alpha }}X_{n+i}=X_{i}.
\end{equation*}%
(see \cite{Hasegawa}). Let $\gamma :I\rightarrow
\mathbb{R}
^{2n+s}(-3s)$ be a unit-speed slant curve with contact angle $\theta $. Let
us denote%
\begin{equation*}
\gamma (t)=\left( \gamma _{1}(t),...,\gamma _{n}(t),\gamma
_{n+1}(t),...,\gamma _{2n}(t),\gamma _{2n+1}(t),...,\gamma _{2n+s}(t)\right)
,
\end{equation*}%
where $t$ is the arc-length parameter. Then $\gamma $ has the tangent vector
field
\begin{equation*}
T=\gamma _{1}^{\prime }\frac{\partial }{\partial x_{1}}+...+\gamma
_{n}^{\prime }\frac{\partial }{\partial x_{n}}+\gamma _{n+1}^{\prime }\frac{%
\partial }{\partial y_{1}}+...+\gamma _{2n}^{\prime }\frac{\partial }{%
\partial y_{n}}+\gamma _{2n+1}^{\prime }\frac{\partial }{\partial z_{1}}%
+...+\gamma _{2n+s}^{\prime }\frac{\partial }{\partial z_{s}},
\end{equation*}%
which can be written as%
\begin{eqnarray*}
T &=&\frac{1}{2}\left[ \gamma _{n+1}^{\prime }X_{1}+...+\gamma _{2n}^{\prime
}X_{n}+\gamma _{1}^{\prime }X_{n+1}+...+\gamma _{n}^{\prime }X_{2n}\right. \\
&&\text{ \ \ \ }+\left( \gamma _{2n+1}^{\prime }-\gamma _{1}^{\prime }\gamma
_{n+1}-...-\gamma _{n}^{\prime }\gamma _{2n}\right) \xi _{1}+... \\
&&\text{ \ \ \ }\left. +\left( \gamma _{2n+s}^{\prime }-\gamma _{1}^{\prime
}\gamma _{n+1}-...-\gamma _{n}^{\prime }\gamma _{2n}\right) \xi _{s}\right] .
\end{eqnarray*}%
Since $\gamma $ is slant curve, we have%
\begin{equation*}
\eta ^{\alpha }(T)=\frac{1}{2}\left( \gamma _{2n+\alpha }^{\prime }-\gamma
_{1}^{\prime }\gamma _{n+1}-...-\gamma _{n}^{\prime }\gamma _{2n}\right)
=\cos \theta
\end{equation*}%
for all $\alpha =1,...,s$. So, we obtain%
\begin{equation}
\gamma _{2n+1}^{\prime }=...=\gamma _{2n+s}^{\prime }=2\cos \theta +\gamma
_{1}^{\prime }\gamma _{n+1}+...+\gamma _{n}^{\prime }\gamma _{2n}.
\label{slant}
\end{equation}%
Since $\gamma $ is a unit-speed, we can write%
\begin{equation}
\left( \gamma _{1}^{\prime }\right) ^{2}+...+\left( \gamma _{2n}^{\prime
}\right) ^{2}=4\left( 1-s\cos ^{2}\theta \right) .  \label{unit}
\end{equation}%
These equations were obtained in our paper \cite{GO-2018}.

Now, our aim is to find parametric equations for slant normal magnetic
curves. So, let us assume that $\gamma :I\rightarrow
\mathbb{R}
^{2n+s}(-3s)$ is a normal magnetic curvature. From the Lorentz equation, we
have%
\begin{equation}
\nabla _{T}T=-q\varphi T\text{,}  \label{1star}
\end{equation}%
where $q\neq 0$ is a constant. Using the Levi-Civita connection, we calculate%
\begin{eqnarray}
\nabla _{T}T &=&\frac{1}{2}\left\{ \left( \gamma _{n+1}^{\prime \prime
}+2s\cos \theta \gamma _{1}^{\prime }\right) X_{1}+...+\left( \gamma
_{2n}^{\prime \prime }+2s\cos \theta \gamma _{n}^{\prime }\right)
X_{n}\right.  \label{2star} \\
&&\left. +\left( \gamma _{1}^{\prime \prime }-2s\cos \theta \gamma
_{n+1}^{\prime }\right) X_{n+1}+...+\left( \gamma _{n}^{\prime \prime
}-2s\cos \theta \gamma _{2n}^{\prime }\right) X_{2n}\right\}  \notag
\end{eqnarray}%
and%
\begin{equation}
\varphi T=\frac{1}{2}\left\{ -\gamma _{1}^{\prime }X_{1}-...-\gamma
_{n}^{\prime }X_{n}+\gamma _{n+1}^{\prime }X_{n+1}+...+\gamma _{2n}^{\prime
}X_{2n}\right\} .  \label{3star}
\end{equation}%
From equations (\ref{1star}), (\ref{2star}) and (\ref{3star}), we have%
\begin{gather*}
\frac{\gamma _{n+1}^{\prime \prime }+2s\cos \theta \gamma _{1}^{\prime }}{%
-\gamma _{1}^{\prime }}=...=\frac{\gamma _{2n}^{\prime \prime }+2s\cos
\theta \gamma _{n}^{\prime }}{-\gamma _{n}^{\prime }} \\
=\frac{\gamma _{1}^{\prime \prime }-2s\cos \theta \gamma _{n+1}^{\prime }}{%
\gamma _{n+1}^{\prime }}=...=\frac{\gamma _{n}^{\prime \prime }-2s\cos
\theta \gamma _{2n}^{\prime }}{\gamma _{2n}^{\prime }}=-q,
\end{gather*}%
which is equivalent to%
\begin{equation}
\frac{\gamma _{n+1}^{\prime \prime }}{-\gamma _{1}^{\prime }}=...=\frac{%
\gamma _{2n}^{\prime \prime }}{-\gamma _{n}^{\prime }}=\frac{\gamma
_{1}^{\prime \prime }}{\gamma _{n+1}^{\prime }}=...=\frac{\gamma
_{n}^{\prime \prime }}{\gamma _{2n}^{\prime }}=\lambda ,
\label{mainequation}
\end{equation}%
where we denote $\lambda =-q+2s\cos \theta $. Firstly, let us assume $%
\lambda \neq 0$. From equation (\ref{mainequation}), if we select pairs%
\begin{equation*}
\frac{\gamma _{n+1}^{\prime \prime }}{-\gamma _{1}^{\prime }}=\frac{\gamma
_{1}^{\prime \prime }}{\gamma _{n+1}^{\prime }},...,\frac{\gamma
_{2n}^{\prime \prime }}{-\gamma _{n}^{\prime }}=\frac{\gamma _{n}^{\prime
\prime }}{\gamma _{2n}^{\prime }},
\end{equation*}%
solving ODEs, we have%
\begin{equation*}
\left( \gamma _{1}^{\prime }\right) ^{2}+\left( \gamma _{n+1}^{\prime
}\right) ^{2}=c_{1}^{2},...,\left( \gamma _{n}^{\prime }\right) ^{2}+\left(
\gamma _{2n}^{\prime }\right) ^{2}=c_{n}^{2},
\end{equation*}%
where $c_{1},...,c_{n}$ are arbitrary constants. Thus, we can write%
\begin{eqnarray}
\gamma _{1}^{\prime } &=&c_{1}\cos f_{1},...,\gamma _{n}^{\prime }=c_{n}\cos
f_{n},  \label{equa1} \\
\gamma _{n+1}^{\prime } &=&c_{1}\sin f_{1},...,\gamma _{2n}^{\prime
}=c_{n}\sin f_{n},  \notag
\end{eqnarray}%
where $f_{1},...,f_{n}$ are differentiable functions on $I$. From (\ref%
{mainequation}) and (\ref{equa1}), we find%
\begin{equation*}
f_{1}^{\prime }=...=f_{n}^{\prime }=-\lambda ,
\end{equation*}%
which gives us%
\begin{equation*}
f_{i}=-\lambda t+a_{i},\text{ }i=1,2,...,n
\end{equation*}%
where $a_{1},...,a_{n}$ are arbitrary constants. Now, if we integrate (\ref%
{equa1}), we have%
\begin{equation*}
\gamma _{1}=\frac{c_{1}}{-\lambda }\sin f_{1}+b_{1},...,\gamma _{n}=\frac{%
c_{n}}{-\lambda }\sin f_{n}+b_{n},
\end{equation*}%
\begin{equation*}
\gamma _{n+1}=\frac{c_{1}}{\lambda }\cos f_{1}+d_{1},...,\gamma _{2n}=\frac{%
c_{n}}{\lambda }\cos f_{n}+d_{n},
\end{equation*}%
where $b_{i}$ and $d_{i}$ are arbitrary constants $(i=1,...,n)$. Thus, we get%
\begin{equation*}
\gamma _{1}^{\prime }\gamma _{n+1}+...+\gamma _{n}^{\prime }\gamma _{2n}=%
\overset{n}{\underset{i=1}{\sum }}\left( \frac{c_{i}^{2}}{\lambda }\cos
^{2}f_{i}+c_{i}d_{i}\cos f_{i}\right) \text{.}
\end{equation*}%
Using the last equation with (\ref{slant}), we obtain%
\begin{equation*}
\gamma _{2n+\alpha }^{\prime }=2\cos \theta +\overset{n}{\underset{i=1}{\sum
}}\left( \frac{c_{i}^{2}}{\lambda }\cos ^{2}f_{i}+c_{i}d_{i}\cos
f_{i}\right) ,
\end{equation*}%
where $\alpha =1,...,s$. If we integrate this last equation, we find%
\begin{equation*}
\gamma _{2n+\alpha }=2t\cos \theta -\overset{n}{\underset{i=1}{\sum }}%
\left\{ \frac{c_{i}^{2}}{4\lambda ^{2}}\left[ \sin \left( 2f_{i}\right)
+2f_{i}\right] +\frac{c_{i}d_{i}}{\lambda }\sin f_{i}\right\} +h_{\alpha },
\end{equation*}%
for $\alpha =1,...,s$ and $h_{1},...,h_{s}$ are arbitrary constants.
Moreover, from (\ref{unit}) and (\ref{equa1}), we have%
\begin{equation}
c_{1}^{2}+...+c_{n}^{2}=4\left( 1-s\cos ^{2}\theta \right) .
\label{unitlamdanonzero}
\end{equation}%
Thus, we have just finished the case $\lambda \neq 0$.

Secondly, let $\lambda =0$. In this case, we have%
\begin{equation*}
\gamma _{1}^{\prime \prime }=\gamma _{2}^{\prime \prime }=...=\gamma
_{2n}^{\prime \prime }=0,
\end{equation*}%
which gives us%
\begin{equation*}
\gamma _{i}=c_{i}t+d_{i},\text{ }
\end{equation*}%
for $i=1,...,2n,$ where $c_{i}$ and $d_{i}$ are arbitrary constants. Using
the last equation, we calculate%
\begin{equation*}
\gamma _{1}^{\prime }\gamma _{n+1}+...+\gamma _{n}^{\prime }\gamma _{2n}=%
\overset{n}{\underset{i=1}{\sum }}c_{i}\left( c_{n+i}t+d_{n+i}\right) .
\end{equation*}%
So, equation (\ref{slant}) becomes%
\begin{equation*}
\gamma _{2n+\alpha }^{\prime }=2\cos \theta +\overset{n}{\underset{i=1}{\sum
}}c_{i}\left( c_{n+i}t+d_{n+i}\right) ,
\end{equation*}%
which gives us%
\begin{equation*}
\gamma _{2n+\alpha }=2t\cos \theta +\overset{n}{\underset{i=1}{\sum }}%
c_{i}\left( \frac{c_{n+i}}{2}t^{2}+d_{n+i}t\right) +h_{\alpha },
\end{equation*}%
where $h_{\alpha }$ are arbitrary constants for $\alpha =1,...,s$. Since $%
\gamma $ is unit-speed, from (\ref{unit}), we have%
\begin{equation*}
c_{1}^{2}+...+c_{2n}^{2}=4\left( 1-s\cos ^{2}\theta \right) .
\end{equation*}%
To sum up, we give the following Theorem:

\begin{theorem}
The slant normal magnetic curves on $%
\mathbb{R}
^{2n+s}(-3s)$ satisfying the Lorentz equation $\nabla _{T}T=-q\varphi T$
have the parametric equations

a)
\begin{equation*}
\gamma _{i}(t)=\frac{c_{i}}{-\lambda }\sin f_{i}(t)+b_{i},
\end{equation*}%
\begin{equation*}
\gamma _{n+i}(t)=\frac{c_{i}}{\lambda }\cos f_{i}(t)+d_{i},
\end{equation*}%
\begin{equation*}
\gamma _{2n+\alpha }(t)=2t\cos \theta -\overset{n}{\underset{i=1}{\sum }}%
\left\{ \frac{c_{i}^{2}}{4\lambda ^{2}}\left[ \sin \left( 2f_{i}(t)\right)
+2f_{i}(t)\right] +\frac{c_{i}d_{i}}{\lambda }\sin f_{i}(t)\right\}
+h_{\alpha },
\end{equation*}%
\begin{equation*}
f_{i}(t)=-\lambda t+a_{i},
\end{equation*}%
\begin{equation*}
\alpha =1,...,s,\text{ }i=1,2,...,n,
\end{equation*}%
\begin{equation*}
\lambda =-q+2s\cos \theta \neq 0
\end{equation*}%
where $a_{i},$ $b_{i},$ $c_{i},$ $d_{i}$ and $h_{\alpha }$ are arbitrary
constants such that $c_{i}$ satisfies%
\begin{equation*}
c_{1}^{2}+...+c_{n}^{2}=4\left( 1-s\cos ^{2}\theta \right) ;
\end{equation*}%
or

b)
\begin{equation*}
\gamma _{i}(t)=c_{i}t+d_{i},
\end{equation*}%
\begin{equation*}
\gamma _{2n+\alpha }(t)=2t\cos \theta +\overset{n}{\underset{i=1}{\sum }}%
c_{i}\left( \frac{c_{n+i}}{2}t^{2}+d_{n+i}t\right) +h_{\alpha },
\end{equation*}%
\begin{equation*}
\alpha =1,...,s,\text{ }i=1,...,2n,
\end{equation*}%
where $c_{i},$ $d_{i}$ and $h_{\alpha }$ are arbitrary constants such that $%
c_{i}$ satisfies%
\begin{equation*}
c_{1}^{2}+...+c_{2n}^{2}=4\left( 1-s\cos ^{2}\theta \right) .
\end{equation*}%
In both cases, $q\neq 0$ is a constant and $\theta $ denotes the constant
contact angle satisfying $\left\vert \cos \theta \right\vert \leq \frac{1}{%
\sqrt{s}}.$
\end{theorem}

In particular, if $s=1$, we obtain Theorem 3.5 in \cite{DIMN-2015}.

\textbf{Acknowledgements.} This work is financially supported by Balikesir Research Grant no. BAP 2018/016.


\begin{thebibliography}{99}
\bibitem{Adachi-1996} Adachi ,T.: Curvature bound and trajectories for
magnetic fields on a Hadamard surface. Tsukuba J. Math. \textbf{20},
225--230, (1996).

\bibitem{Blair-2010} Blair, \ D. E.: Riemannian geometry of contact and
symplectic manifolds, Second edition. Progress in Mathematics, 203.
Birkhauser Boston, Inc., Boston, MA, (2010).

\bibitem{Blair-1970} Blair, D. E.: Geometry of manifolds with structural
group $U(n)\times O(s)$. J. Differential Geometry, \textbf{4}, 155-167,
(1970).

\bibitem{BRCF} Barros M., Romero, A.,~Cabrerizo, J.~L.,~Fern\'{a}ndez, M.:~
The Gauss-Landau-Hall problem on Riemannian surfaces. J. Math. Phys. \textbf{%
46}, no. 11, 112905, 15 pp, (2005).

\bibitem{CB-1994} Baikoussis, C., Blair, D. E.: On Legendre curves in
contact $3$-manifolds. Geom. Dedicata \textbf{49}, 135--142 (1994).

\bibitem{CFG-2009} Cabrerizo, J. L., Fernandez M. and Gomez, J. S.: On the
existence of almost contact structure and the contact magnetic field. Acta
Math. Hungar. \textbf{125}, 191-199, (2009).

\bibitem{CMP-2015} Calvaruso, G., Munteanu, M. I., Perrone, A.: Killing
magnetic curves in three-dimensional almost paracontact manifolds. J. Math.
Anal. Appl. \textbf{426}, no. 1, 423--439, (2015).

\bibitem{CIL} Cho, J. T., Inoguchi, J., Lee, J.E.: On slant curves in
Sasakian $3$-manifolds. Bull. Austral. Math. Soc. \textbf{74}, 359--367
(2006).

\bibitem{Comtet-1987} Comtet, A.:~ On the Landau levels on the hyperbolic
plane. Ann. Physics \textbf{173}, 185-209, (1987).

\bibitem{DIMN-2015} Dru\c{t}\u{a}-Romaniuc, S. L., Inoguchi, J., Munteanu,
M. I., Nistor, A. I.: Magnetic curves in Sasakian manifolds. Journal of
Nonlinear Mathematical Physics, \textbf{22}, 428-447, (2015).

\bibitem{DIMN-2016} Dru\c{t}\u{a}-Romaniuc, S. L., Inoguchi, J., Munteanu,
M. I., Nistor, A. I.: \ Magnetic curves in cosymplectic manifolds. Rep.
Math. Phys. \textbf{78}, 33-48, (2016).

\bibitem{GO-2018} G\"{u}ven\c{c}, \c{S}., \"{O}zg\"{u}r, C.: On slant curves
in $S$-manifolds. Commun. Korean Math. Soc. \textbf{33}, No. 1, pp. 293-303,
(2018).

\bibitem{Hasegawa} Hasegawa, I., Okuyama, Y., Abe, T.: On $p$-th Sasakian
manifolds. J. Hokkaido Univ. Ed. Sect. II A, \textbf{37}, no. 1, 1--16,
(1986).

\bibitem{IM-2017} Inoguchi J., Munteanu, M. I.: Periodic magnetic curves in
Berger spheres. Tohoku Math. J. \textbf{69}, 113-128, (2017).

\bibitem{JMN-2015} Jleli, M., Munteanu, M. I., Nistor, A. I.: Magnetic
trajectories in an almost contact metric manifold $\mathbb{R}^{2N+1}$.
Results Math. \textbf{67}, 125-134, (2015).

\bibitem{JM-2015} Jleli, M., Munteanu, M. I.: Magnetic curves on flat
para-Kahler manifolds. Turkish J. Math. \textbf{39}, 963-969, (2015).

\bibitem{MN-2012} Munteanu, M. I., Nistor, A. I.: The classification of
Killing magnetic curves in $S^{2}\times R$. J. Geom. Phys. \textbf{62},
170-182,\textbf{\ }(2012).

\bibitem{MN-2014} Munteanu, M. I., Nistor, A. I.: A note on magnetic curves
on $S^{2n+1}$. C. R. Math. Acad. Sci. Paris \textbf{352}, 447-449, (2014).

\bibitem{MN-2017} Munteanu, M. I., Nistor, A. I.: On some closed magnetic
curves on a $3$-torus. Math. Phys. Anal. Geom. \textbf{20}, no. 2, Art. 8,
13 pp, (2017).

\bibitem{Nak-1966} Nakagawa, H.: On framed $f$-manifolds. Kodai Math. Sem.
Rep. \textbf{18,} 293-306 (1966).

\bibitem{OG-2014} \"{O}zg\"{u}r, C., G\"{u}ven\c{c}, \c{S}.: On biharmonic
Legendre curves in $S$-space forms. Turkish J. Math. \textbf{38}, no. 3,
454--461 (2014).

\bibitem{Ozgur-2017} \"{O}zg\"{u}r, C.: On magnetic curves in the
3-dimensional Heisenberg group. Proceedings of the Institute of Mathematics
and Mechanics, National Academy of Sciences of Azerbaijan, \textbf{43}, 2,
278-286, (2017).

\bibitem{OGY-2015} \"{O}zdemir, Z., G\"{o}k, I., Yayl\i , Y., Ekmekci, N.:
Notes on magnetic curves in 3D semi-Riemannian manifolds. Turkish J. Math.
\textbf{39}, 412-426, (2015).

\bibitem{Vanzura-1972} Vanzura, J.: Almost $r$-contact structures. Ann.
Scuola Norm. Sup. Pisa (3) \textbf{26,} 97-115 (1972).

\bibitem{YK-1984} Yano, K., Kon, M.: Structures on Manifolds. Series in Pure
Mathematics, \textbf{3}. Singapore. World Scientific Publishing Co. 1984.
\end{thebibliography}
\end{document}